\documentclass[reqno]{amsart}
\usepackage{latexsym,amsmath,amsfonts,amssymb,amsthm,amscd,enumitem,comment}
\usepackage[usenames,dvipsnames]{xcolor}
\usepackage[colorlinks=true,linkcolor=blue]{hyperref} 
\usepackage{hyperref}
\newtheorem{theorem}{Theorem}[section]
\newtheorem{lemma}[theorem]{Lemma}
\newtheorem{proposition}[theorem]{Proposition}
\newtheorem{corollary}[theorem]{Corollary}
\theoremstyle{definition}
\newtheorem{definition}[theorem]{Definition}

\newtheorem{nota}[theorem]{Remark}

\numberwithin{equation}{section}

\newcommand\JM{Mierczy\'nski}
\newcommand\R{\ensuremath{\mathbb{R}}}

\newcommand\NN{\ensuremath{\mathbb{N}}}
\newcommand\PP{\ensuremath{\mathbb{P}}}
\newcommand\TT{\ensuremath{\mathbb{T}}}
\newcommand{\w}{\omega}

\newcommand{\OFP}{\ensuremath{(\Omega,\mathfrak{F},\PP)}}

\newcommand{\norm}[1]{\ensuremath{\lVert#1\rVert}}
\newcommand{\normL}[1]{\ensuremath{\lVert#1\rVert_{L}}}
\newcommand{\normC}[1]{\ensuremath{\lVert#1\rVert_{C}}}

\newcommand{\mlsps}{measurable linear skew\nobreakdash-\hspace{0pt}product
semidynamical system}
\newcommand{\mlspss}{measurable linear skew\nobreakdash-\hspace{0pt}product
semidynamical systems}

\DeclareMathOperator{\lnplus}{ln^{+}}

\begin{document}
\title[Lyapunov Exponents and Oseledets Decomposition]{Lyapunov Exponents and Oseledets Decomposition in Random Dynamical Systems Generated by Systems of Delay Differential Equations}
\author[J.~Mierczy\'nski]{Janusz Mierczy\'nski}
\address[Janusz Mierczy\'nski]{Faculty of Pure and Applied Mathematics,
Wroc{\l}aw University of Science and Technology,
Wybrze\.ze Wyspia\'nskiego 27, PL-50-370 Wroc{\l}aw, Poland.}
\email[Janusz Mierczy\'nski]{Janusz.Mierczynski@pwr.edu.pl}
\address[Sylvia Novo, Rafael Obaya]
{Departamento de Matem\'{a}tica Aplicada, Universidad de
Valladolid, Paseo del Cauce 59, 47011 Valladolid, Spain.}
\email[Sylvia Novo]{sylnov@wmatem.eis.uva.es}
\email[Rafael Obaya]{rafoba@wmatem.eis.uva.es}
\thanks{The first author is supported by the NCN grant Maestro 2013/08/A/ST1/00275 and the last two authors are partly supported by MEC (Spain)
under project MTM2015-66330-P and EU Marie-Sk\l odowska-Curie ITN Critical Transitions in
Complex Systems (H2020-MSCA-ITN-2014 643073 CRITICS)}
\author[S.~Novo]{Sylvia Novo}
\author[R.~Obaya]{Rafael Obaya}
\subjclass[2010]{Primary: 37H15, 37L55, 34K06. Secondary: 37A30, 60H25.}
\date{}
\begin{abstract}
Linear skew-product semidynamical systems generated by random systems of delay differential equations are considered, both on a space of continuous functions as~well as on a space of $p$-summable functions.  The main result states that in both cases, the Lyapunov exponents are identical, and that the Oseledets decompositions are related by natural embeddings.
\end{abstract}
\keywords{Random dynamical systems, Linear systems of delay differential equations, Lyapunov exponents, Oseledets decomposition}
\maketitle
\section{Introduction}\label{sect:introduction}
The theory of linear random skew-product semidynamical systems has become a powerful tool in the investigation of random linear parabolic PDEs of second order driven by a measurable dynamical system on a probability space.  In~particular, when the solution operator is compact (and that holds if the domain is bounded) then, assuming the summability of the coefficients of the PDE, we have an \emph{Oseledets decomposition}:  the separable Banach space decomposes into a countable direct sum of invariant measurable families of finite-dimensional vector subspaces which can be characterized as corresponding to solutions defined on the whole real line having given logarithmic growth rates (\emph{Lyapunov exponents}) both in the future and in the past (plus, possibly, an invariant measurable family corresponding to solutions having logarithmic growth rates $-\infty$).  See Lian and Lu's monograph~\cite{Lian-Lu}.
\par\smallskip
When we consider systems of linear random delay differential equations
\begin{equation*}
  z'(t) = A(\theta_t \omega)\, z(t) + B(\theta_t \omega) \, z(t - 1),
\end{equation*}
any ``natural" space on which we define a linear skew-product semidynamical system must contain (or, at~least, be equal to) a space consisting of functions defined on $[-1, 0]$ and taking values on $\R^N$.  In~general, there is no hope that the solution operator is compact.  But it is compact after some time, so this is not a big obstacle.
A more important thing is that, generally, the solution operator is not injective.  This makes it impossible to directly apply the results contained in~\cite{Lian-Lu}.
\par
\smallskip
It is natural to work in the framework of \emph{semi-invertible Oseledets theorems}:  the metric dynamical system on the base space is invertible, but the operators between fibers are not necessarily injective.
\par
\smallskip
For such systems, Doan proved in his dissertation~\cite{Doan} the existence of an \emph{Oseledets filtration}: an invariant measurable filtration by  finite-codimensional vector subspaces such that the solutions corresponding to the set difference of two subsequent subspaces have logarithmic rates of growth equal to a given Lyapunov exponent.  Starting from Doan's results,  Gonz\'alez-Tokman and Quas~\cite{GTQu} proved that there exists an Oseledets splitting, provided only that the fibers are separable Banach spaces and that the base space is a Lebesgue space. Indeed, in an earlier paper by Froyland \emph{et al.}~\cite{FrLlQu} an Oseledets splitting was obtained, but under an additional assumption that the base space is a Borel subset of a separable complete metric space with the $\sigma$-algebra of Borel sets and with a Borel probability measure.
\par
\smallskip
The above result should be considered sufficient for our purposes:  $C([-1,0], \R^N)$ appears to be the natural, at first sight, Banach space for which the solution operator satisfies all the axioms of a skew-product random semidynamical system.  Such a Banach space is separable, and there are no difficulties.
\par\smallskip
However, one should remember that we sometimes need to calculate (at~least, to estimate) the Lyapunov exponents.  As shown in  Calzada \emph{et al.}~\cite{COS}, one needs a Hilbert space, more precisely, the space $L^2([-1,0], \R^N, \mu_0)$, with $\mu_0 = \delta_0 + l$, where $l$ is the Lebesgue measure on $[-1, 0]$, is a natural choice here.  In such a case, one can use results from Gonz\'alez-Tokman and Quas~\cite{GTQu2}: an Oseledets decomposition is proved there for reflexive separable Banach spaces. In general, good geometric properties for the Banach spaces provide a more constructive version of the theory. \JM\ and Shen~\cite{MiShPart1} and \JM\ \emph{et al.}~\cite{MNO}  prove, under adequate dynamical assumptions, the existence of a principal Floquet subspace and a generalized exponential separation decomposition when the fiber is a separable Banach space with separable dual.
\par\smallskip
This has to do with the dual skew-product semidynamical systems.  In the case of ordinary differential equations, or parabolic partial differential equations of second order, such dual skew-product systems are generated by adjoint equations.  Then, the adjoint equation has the same properties as the original equation, and in many cases one needs only to prove ``one half" of a theorem (for example, the existence of an Oseledets filtration, whereas the other half can be given by applying the theorem to the skew-product system generated by the adjoint equation; for a similar approach see Section~3 in \JM\ and Shen~\cite{MiShPart3}).
\par\smallskip
However, this is not the case for delay differential equations.  To be sure, there exists a well-defined ``abstract" dual skew-product semidynamical system, but, at~least in the case of $C([-1, 0], \R^N)$, it is not generated by anything resembling an adjoint equation.  For generation of the dual system by an adjoint equation (sort~of), see Delfour and Mitter~\cite{DM75}.
\par
\medskip
The paper is organized as follows. Section~\ref{subsect:prelim} contains preliminaries and explains notions used throughout the rest of the sections.  In Section~\ref{subsect:Oseledets} a definition of an Oseledets decomposition for a \mlsps\ is given, and, under appropriate assumptions for our purposes, some theorems of existence are explained
\par\smallskip
Section 4 is devoted to showing that linear systems of delay differential equations generate \mlsps s  when we take as our fiber both $C([-1, 0], \R^n)$ and $\R^N \times L_p([-1, 0], \R^N)$, and some measurability and summability assumptions on the coefficients are considered.
\par\smallskip
The main results of the paper are contained in Section~\ref{sect:Lyapunov}.  It is shown that for both spaces, the Lyapunov exponents are the same, and that the Oseledets decomposition are related by natural embeddings. The importance of these results is that the geometrical methods of construction for the Oseledets subspaces, obtained in~\cite{GTQu2} for reflexive separable Banach spaces, as well as the estimates of Lyapunov exponents, can be applied on $\R^N \times L_p([-1, 0], \R^N)$ and then translated to $C([-1, 0], \R^N)$ by embedding.
\section{Preliminaries}\label{subsect:prelim}
Let  $(Y, d)$ be a metric space, $B(y; \epsilon)$ denotes the closed ball in $Y$ centered at $y \in Y$ with radius $\epsilon > 0$, and $\mathfrak{B}(Y)$  stands for the
$\sigma$\nobreakdash-\hspace{0pt}algebra of all Borel subsets of $Y$.
For a compact metric space $Z$ and a Banach space $X$, $C(Z, X)$ denotes the Banach space of continuous functions from $Z$ into $X$, with the supremum norm. For Banach spaces $X_1$, $X_2$, $\mathcal{L}(X_1, X_2)$ stands for the Banach space of bounded linear mappings from $X_1$ into $X_2$, endowed with the standard norm.  Instead~of $\mathcal{L}(X, X)$ we write $\mathcal{L}(X)$.
\par\smallskip
A {\em probability space\/} is a triple $\OFP$, where $\Omega$ is a nonempty set, $\mathfrak{F}$ is a $\sigma$\nobreakdash-\hspace{0pt}algebra of subsets of $\Omega$, and $\PP$ is a probability measure defined for all $F \in \mathfrak{F}$.  We always assume that the measure $\PP$ is complete.
\par\smallskip
A \emph{measurable dynamical system} on the probability space $\OFP$ is a $(\mathfrak{B}(\R) \otimes \mathfrak{F},\mathfrak{F})$\nobreakdash-\hspace{0pt}measurable mapping $\theta\colon\R\times \Omega\to \Omega$ such that
\begin{itemize}
\item
$\theta(0,\w)=\w$ for any $\w\in\Omega$,
\item
$\theta(t_1+t_2,w)=\theta(t_2,\theta(t_1,\w))$ for any $t_1$, $t_2\in\R$ and any $\w \in\Omega$.
\end{itemize}
We write $\theta(t,\w)$ as $\theta_t\w$. Also, we usually denote measurable dynamical systems by $(\OFP,(\theta_{t})_{t \in \R})$ or simply by $(\theta_{t})_{t \in \R}$.\par\smallskip
A \emph{metric dynamical system} is a measurable dynamical system $(\OFP,(\theta_{t})_{t \in \R})$
such that for each $t\in\R$  the mapping $\theta_t\colon \Omega\to\Omega$ is $\PP$-preserving (i.e., $\PP(\theta_t^{-1}(F))=\PP(F)$ for any $F\in\mathfrak{F}$ and $t\in\R$). It is said to be {\em ergodic\/} if for any invariant $F \in \mathfrak{F}$, either $\PP(F) = 1$ or $\PP(F) = 0$.
\par\smallskip
We write $\R^{+}$ for $[0, \infty)$.  By a {\em measurable linear skew-product semidynamical system or semiflow},
$\Phi = \allowbreak ((U_\w(t))_{\w \in \Omega, t \in \R^{+}}, \allowbreak (\theta_t)_{t\in\R})$ on a Banach space $X$ covering a metric dynamical system $(\theta_{t})_{t \in \R}$ we understand a $(\mathfrak{B}(\R^{+}) \otimes \mathfrak{F} \otimes \mathfrak{B}(X), \mathfrak{B}(X))$\nobreakdash-\hspace{0pt}measurable
mapping
\begin{equation*}
[\, \R^{+} \times \Omega \times X \ni (t,\w,u) \mapsto U_{\w}(t)\,u \in X \,]
\end{equation*}
satisfying
\begin{align}
&U_{\w}(0) = \mathrm{Id}_{X} \quad & \textrm{for each }\,\w  \in \Omega, \label{eq-identity}\\
&U_{\theta_{s}\w}(t) \circ U_{\w}(s)= U_{\w}(t+s) \qquad &\textrm{for each } \,\w \in \Omega \textrm{ and }  t,\,s \in \R^{+},\label{eq-cocycle}\\
&[\, X \ni u \mapsto U_{\w}(t)u \in X \,] \in \mathcal{L}(X) & \textrm{for each }\,\w \in \Omega \textrm{ and } t \in \R^{+}.\nonumber
\end{align}
Sometimes we write simply $\Phi = ((U_\w(t)), \allowbreak (\theta_t))$. Eq.~\eqref{eq-cocycle} is called the {\em cocycle property\/}.
\par\smallskip
\par
We use also the {\em semiprocess\/} notation:  for $\omega \in \Omega$ and $0 \le s \le t$,  $U_{\omega}(t, s)$ stands for $U_{\theta_{s}\omega}(t - s)$.  Then $U_{\omega}(t) = U_{\omega}(t, 0)$, and~\eqref{eq-cocycle} takes the form
\begin{equation*}
U_{\omega}(t, \tau) = U_{\omega}(t, s) \circ U_{\omega}(s, \tau) \qquad \textrm{ for each } \omega \in \Omega, \ 0 \le \tau \le s \le t.
\end{equation*}
\par
Given $\omega \in \Omega$ and $u \in X$,  the {\em positive semiorbit\/} passing through $(\omega, u)$ is the  $(\mathfrak{B}([0, \infty), \mathfrak{B}(X))$-measurable mapping
\begin{equation*}
\bigl[ \, [0, \infty) \ni t \mapsto U_{\omega}(t) u \in X \, \bigr].
\end{equation*}
\par
Given $\omega \in \Omega$ and $u \in X$, a {\em negative semiorbit\/} passing through $(\omega, u)$ is a $(\mathfrak{B}((-\infty,0], \mathfrak{B}(X))$-measurable mapping $\tilde{u} \colon (-\infty,0] \to X$
such that:
\begin{itemize}
\item
$\tilde{u}(0) = u$;
\item
$\tilde{u}(t + s)=U_{\theta_{s}\omega}(t) \,\tilde{u}(s)$ for any $s \le 0$,  $t \ge 0$  such that $s + t \le 0$.
\end{itemize}
For $(\omega, u)$ a negative semiorbit need not exist, and, if it exists, it need not be unique.
  A {\em full or entire orbit\/} passing through $(\omega, u)\in\Omega\times X$ is a $(\mathfrak{B}(\R, \mathfrak{B}(X))$-measurable mapping $\tilde{u} \colon \R \to X$
such that:
\begin{itemize}
\item
$\tilde{u}(0) = u$;
\item
$\tilde{u}(t + s)=U_{\theta_{s}\omega}(t) \,\tilde{u}(s)$ for any $s \in \R$,  $t \ge 0$.
\end{itemize}
Let $\Omega_0 \in \mathfrak{F}$.
A family $\{E(\omega)\}_{\omega \in \Omega_0}$ of $l$\nobreakdash-\hspace{0pt}dimensional vector subspaces of $X$ is {\em measurable\/} if there are $(\mathfrak{F},
\mathfrak{B}(X))$-measurable functions $v_1, \dots, v_l \colon \Omega_0 \to X$ such that $\{v_1(\omega), \dots, v_l(\omega)\}$ forms a basis of $E(\omega)$ for each $\omega \in \Omega_0$.
\par \smallskip
Let $\{E(\omega)\}_{\omega \in \Omega_0}$ be a family of $l$\nobreakdash-\hspace{0pt}dimensional vector subspaces of $X$, and let $\{F(\omega)\}_{\omega \in \Omega_0}$ be a family of $l$\nobreakdash-\hspace{0pt}codimensional closed vector subspaces of $X$, such that $E(\omega) \oplus F(\omega) = X$ for all $\omega \in \Omega_0$.  We define the {\em family of projections associated with the decomposition\/} $E(\omega) \oplus F(\omega) = X$ as $\{P(\omega)\}_{\omega \in \Omega_0}$, where $P(\omega)$ is the linear projection of $X$ onto $F(\omega)$ along $E(\omega)$, for each $\omega \in \Omega_0$.
\par\smallskip
The family of projections associated with the decomposition $E(\omega) \oplus F(\omega) = X$ is called {\em strongly measurable\/} if for each $u \in X$ the mapping $[\, \Omega_0 \ni \omega \mapsto P(\omega)u \in X \,]$ is $(\mathfrak{F}, \mathfrak{B}(X))$\nobreakdash-\hspace{0pt}measurable.
\par\smallskip
We say that the decomposition $E(\omega) \oplus F(\omega) = X$, with
$\{E(\omega)\}_{\omega \in \Omega_0}$ finite\nobreakdash-\hspace{0pt}dimensional, is {\em invariant\/} if $\Omega_0$ is invariant, $U_{\omega}(t)E(\omega) = E(\theta_{t}\omega)$
and $U_{\omega}(t)F(\omega) \subset F(\theta_{t}\omega)$, for each $t \in \TT^{+}$.
\par\smallskip
A strongly measurable family of projections associated with the invariant decomposition $E(\omega) \oplus F(\omega) = X$ is referred to as {\em tempered\/} if
\begin{equation*}
\lim\limits_{t \to \pm\infty} \frac{\ln{\norm{P(\theta_{t}\omega)}}}{t} = 0 \qquad \PP\text{-a.e. on }\Omega_0.
\end{equation*}
\section{Oseledets decomposition}\label{subsect:Oseledets}
Let $\Phi = ((U_\w(t)), \allowbreak (\theta_t))$ be a \mlsps\ covering an ergodic metric dynamical system $(\OFP,(\theta_{t})_{t \in \R})$ with $\PP$ complete. We assume throughout the present section that
\begin{enumerate}[label=(\textbf{O\arabic*}),leftmargin=27pt]
\item\label{O1}  the functions
\begin{equation*}
\bigl[ \,\omega \mapsto \sup\limits_{0 \le s \le 1} {\lnplus{\norm{U_{\omega}(s)}}} ) \, \bigr]\text{ and }
\bigl[ \, \omega \mapsto \sup\limits_{0 \le s \le 1}
{\lnplus{\norm{U_{\theta_{s}\omega}(1-s)}}}  \, \bigr]
\end{equation*}
belong to $L_1\OFP$.
\end{enumerate}
\par\smallskip
Then it follows from the Kingman subadditive ergodic theorem that there exists $\lambda_{\mathrm{top}} \in [-\infty, \infty)$ such that
\begin{equation*}
\lim\limits_{t \to \infty} \frac{\ln{\norm{U_{\omega}(t)}}}{t} = \lambda_{\mathrm{top}}
\end{equation*}
for $\PP$-a.e.\ $\omega \in \Omega$, which is referred to as the {\em top Lyapunov exponent\/} of $\Phi$.\par\smallskip
We will also assume that
\begin{enumerate}[label=(\textbf{O2}),leftmargin=27pt]
\item\label{O2} $\lambda_{\mathrm{top}} > -\infty$.
\end{enumerate}
\begin{definition}\label{defi:Oseledets} $\Phi$ admits an {\em Oseledets decomposition\/} if there exists an invariant subset $\Omega_0 \subset \Omega$, $\PP(\Omega_0) = 1$, with the property that one of the following mutually exclusive cases, \textup{(I)} or \textup{(II)}, holds:\par\smallskip
\begin{enumerate}[leftmargin=20pt]
\item[\textup{(I)}]
There are $k$  real numbers $\lambda_1 = \lambda_{\mathrm{top}} > \cdots > \lambda_k$, called  the {\em Lyapunov exponents\/} for $\Phi$, $k$ measurable families $\{E_1(\omega)\}_{\omega\in\Omega_0}$, \ldots, $\{E_k(\omega)\}_{\omega \in \Omega_0}$ of finite dimensional vector subspaces, and a family $\{F_{\infty}(\omega)\}_{\omega\in \Omega_0}$ of closed vector subspaces of finite codimension such that\par\smallskip
\begin{itemize}
\item[\textup{(a)}] for $i = 1,\dots,k$, any $\omega\in\Omega_0$ and $t \ge 0$
\[U_{\omega}(t) E_{i}(\omega) = E_{i}(\theta_t\omega )  \quad\text{and}\quad U_{\omega}(t) F_{\infty}(\omega) \subset F_{\infty}(\theta_t\omega)\,;\]
\item[\textup{(b)}]
$E_1(\omega) \oplus \ldots \oplus E_k(\omega) \oplus F_{\infty}(\omega) = X$ for any $\omega \in \Omega_0$; we write
\[\hspace{1.8cm}  F_i(\omega) := \bigoplus\limits_{j=i+1}^{k} E_{j}(\omega) \oplus F_{\infty}(\omega) \text{ for } i = 1, \dots, k - 1,  \text{ and } F_0(\omega) : = X\,.\] In particular, $F_i(\omega) = E_{i+1}(\omega) \oplus F_{i+1}(\omega)$ for $i = 0, 1, \dots, k - 2\,$;\vspace{.1cm}
\item[\textup{(c)}]
for $i = 1,\dots,k-1$, the families of projections associated with the de\-compositions $\Bigl(\bigoplus\limits_{j=1}^{i} E_{j}(\omega)\Bigr) \oplus F_i(\omega) = X$ and $\Bigl(\bigoplus\limits_{j=1}^{k} E_{j}(\omega)\Bigr) \oplus F_{\infty}(\omega) = X$ are strongly measurable and tempered;\vspace{.1cm}
\item[\textup{(d)}] for $i = 1,\dots, k\,$, any $\omega \in \Omega_0$ and any nonzero $u \in E_i(\omega)$
\[\displaystyle
\lim_{t\to \infty } \frac{\ln{\norm{U_{\omega}(t)|_{E_i(\omega)}}} }{t} = \lim_{t\to \infty} \frac{\ln{\norm{U_{\omega}(t)\,u}}}{t}  = \lambda_i\,;
\]
\item[\textup{(e)}] for $i = 1,\dots, k\,$, any $\omega \in \Omega_0$ and any $u \in F_{i - 1}(\omega) \setminus F_{i}(\omega)$
\[\displaystyle
\lim_{t\to\infty} \frac{\ln{\norm{U_\omega(t)\,u}}}{t}  = {\lambda}_i\,;
\]
\item[\textup{(f)}] for any $\omega \in \Omega_0$ and any $u \in F_{k - 1}(\omega) \setminus F_{\infty}(\omega)$
\[
\lim_{t\to\infty} \frac{\ln{\norm{U_\omega(t)\,u}}}{t} = {\lambda}_k\,;
\]
\item[\textup{(g)}] for $i = 1,\dots, k$ and any $\omega \in \Omega_0$, a nonzero $u \in F_{i - 1}(\omega)$ belongs to $E_i(\omega)$ if and only if there exists a negative semiorbit $\tilde{u} \colon (-\infty ,0] \to X$ passing through $(\omega, u)$ such that
\begin{equation*}
\lim\limits_{s \to -\infty} \frac{\ln{\norm{\tilde{u}(s)}}}{s}  = \lambda_{i};
\end{equation*}
\item[\textup{(h)}] for any $\omega \in \Omega_0$
\[
\lim\limits_{t\to\infty} \frac{\ln{\norm{U_{\omega}(t)|_{F_{\infty}(\omega)}}} }{t} = -\infty\,.
\]
\end{itemize}
In this case, $\;\{F_1(\omega)\}_{\omega \in \Omega_0}, \ldots, \{F_{k-1}(\omega)\}_{\omega \in \Omega_0}, \{F_{\infty}(\omega)\}_{\omega \in \Omega_0}\;$ is called the {\em Oseledets filtration\/} for $\Phi$.\vspace{.2cm}
\item[\textup{(II)}]
There are a decreasing sequence of real numbers $\lambda_1 = \lambda_{\mathrm{top}}> \cdots > \lambda_i > \lambda_{i+1} > \cdots {}$ with limit $-\infty$, called  the {\em Lyapunov exponents\/} for $\Phi$, countably many measurable families $\{E_i(\omega)\}_{\omega \in
\Omega_0}$, $i\in\NN$, of finite dimensional vector subspaces, and countably many families $\{F_i(\omega)\}_{\omega \in \Omega_0}$, $i\in\NN$,
 of closed vector subspaces of finite codimensions, called the {\em Oseledets filtration\/} for $\Phi$,  such that\vspace{.2cm}
\begin{itemize}
\item[\textup{(a)}] for $i\in\NN$, any $\omega \in \Omega_0$ and $t \ge 0$
\[U_{\omega}(t) E_{i}(\omega) = E_{i}(\theta_t\omega)\quad \text{and}\quad U_{\omega}(t) F_{i}(\omega) \subset F_{i}(\theta_t\omega)\,;\]
\item[\textup{(b)}] for $i\in\NN$ and any $\omega \in \Omega_0$
\[\hspace{2cm} E_1(\omega) \oplus \ldots \oplus E_i(\omega) \oplus F_i(\omega) = X \;\text{ and }\; F_i(\omega) = E_{i+1}(\omega) \oplus F_{i+1}(\omega)\,;\]
\item[\textup{(c)}] for $i\in\NN$,
the families of projections associated with the decompositions
$\Bigl(\bigoplus\limits_{j=1}^{i} E_{j}(\omega)\Bigr) \oplus F_{i}(\omega) = X$ are strongly measurable and tempered;\vspace{.1cm}
\item[\textup{(d)}] for $i\in\NN$,  any $\omega \in \Omega_0$ and any nonzero $u \in E_i(\omega)$
\[
\lim_{t \to \infty} \frac{\ln{\norm{U_{\omega}(t)|_{E_i(\omega)}}}}{t}  = \lim_{t \to \infty} \frac{\ln{\norm{U_{\omega}(t)\,u}}}{t}  = \lambda_i\,;\]
\item[\textup{(e)}] for $i\in\NN$,  any $\omega \in \Omega_0$ and any $u \in F_{i-1}(\omega) \setminus F_{i}(\omega)$
\[
\lim\limits_{t\to\infty} \frac{\ln{\norm{U_\omega(t)\,u}}}{t}  = {\lambda}_i
\]
where $F_{0}(\omega) := X$;\vspace{.1cm}
\item[\textup{(f)}] for $i\in\NN$ and any $\omega \in \Omega_0$, a nonzero $u \in F_{i - 1}(\omega)$ belongs to $E_i(\omega)$ if and only if there exists a negative semiorbit $\tilde{u} \colon (-\infty ,0] \to X$ passing through $(\omega, u)$ such that 
\begin{equation*}
\lim\limits_{s \to -\infty} \frac{\ln{\norm{\tilde{u}(s)}}}{s}  = \lambda_{i};
\end{equation*}
\item[\textup{(g)}] for $i\in\NN$ and any $\omega \in \Omega_0$
\[
\lim\limits_{t\to\infty} \frac{\ln{\norm{U_{\omega}(t)|_{F_i(\omega)}}}}{t}  = \lambda_{i+1}\,;\]
\item[\textup{(h)}] for any $\omega \in \Omega_0$ and a nonzero $u \in \bigcap\limits_{j = 1}^{\infty} F_j(\omega) =: F_{\infty}(\omega)$
\[
\lim_{t\to\infty} \frac{\ln{\norm{U_{\omega}(t)\,u}}}{t}  = - \infty
\,.\]
\end{itemize}
\end{enumerate}
\end{definition}
\par
As a consequence of the existence of an Oseledets decomposition, we easily deduce the following properties.
\begin{proposition}
\label{prop:limit-everywhere}
Assume that $\Phi$ admits an Oseledets decomposition. Then for each $\omega \in \Omega_0$ and each nonzero $u \in X$ the limit
\begin{equation}
\label{eq:true-limit}
\lim_{t \to \infty} \frac{\ln{\norm{U_{\omega}(t)\,u}}}{t}
\end{equation}
exists and equals some Lyapunov exponent $\lambda_i$ or $- \infty$.
\end{proposition}
\begin{proof}
Assume case (I) and fix $\omega \in \Omega_0$ and a nonzero $u\in X$.  If $u$ belongs to $F_{\infty}(\omega)$ then~\eqref{eq:true-limit} equals $- \infty$.   If $u$ does not belong to $F_{\infty}(\omega)$ then there is a $j \in \{1, \ldots, k \}$ such that in the decomposition $E_1(\omega) \oplus \ldots \oplus E_k(\omega) \oplus F_{\infty}(\omega) = X$ the $E_j(\omega)$-coordinate of $u$ is nonzero.  Now, take $i$ to be the smallest such a $j$.  If $i \in \{1, \ldots, k - 1 \}$ one has $u \in F_{i - 1}(\omega) \setminus F_{i}(\omega)$, consequently \eqref{eq:true-limit} equals $\lambda_{i}$.  If $i = k$ then $u \in F_{k-1}(\omega) \setminus F_{\infty}(\omega)$ and \eqref{eq:true-limit} is $\lambda_{k}$.
\par\smallskip
Assume case (II) and fix $\omega \in \Omega_0$ and a nonzero $u \in X$.   If in each decomposition $E_1(\omega) \oplus \ldots \oplus E_i(\omega) \oplus F_{i}(\omega) = X$ the $E_j(\omega)$-coordinates, $j \in \{1, \ldots, i\}$, of $u$ are zero, then $u \in F_{\infty}(\omega)$, so \eqref{eq:true-limit} is $- \infty$.  Otherwise there is an $i \in \NN$ such that in the decomposition $E_1(\omega) \oplus \ldots \oplus E_i(\omega) \oplus F_{i}(\omega) = X$ the $E_j(\omega)$-coordinates, $j \in \{1, \ldots, i -1\}$, of $u$ are zero and the $E_i(\omega)$-coordinate is nonzero, then $u \in F_{i - 1}(\omega) \setminus F_{i}(\omega)$ and \eqref{eq:true-limit} equals $\lambda_{i}$.
\end{proof}
\begin{proposition}\label{prop:full-orbit}
Assume that $\Phi$ admits an Oseledets decomposition.  Then for each $\omega \in \Omega_0$ and each $i = 1,\dots, k$ in case \textup{(I)}, or each $i \in\NN$ in case \textup{(II)}, a nonzero $u \in X$ belongs to $E_i(\omega)$ if and only if there exists a full orbit $\widetilde{u} \colon \R \to X$ passing through $(\omega, u)$ such that
\begin{equation*}
\lim\limits_{t \to \pm\infty} \frac{\ln{\norm{\widetilde{u}(t)}}}{t}  = \lambda_{i}\,.
\end{equation*}
\end{proposition}
\begin{proof}
The necessity is a consequence of (e) or (f), and (g) for the case (I), and  (e) and (f) for the case (II).
Assume now that for some nonzero $u \in X$ there exists a full orbit with the above properties.  It follows from (e) or (f) in the first case,  or (e) in the second one, that $u$ belongs to $F_{i-1}(\omega)$ for some $i = 1, \dots, k$, or some $i \in\NN$ respectively.  Now we need to apply (I)(g) or (II)(f) to finish the proof.
\end{proof}
It follows from the results in~\cite{Lian-Lu} that if $X$ is separable and  $U_{\omega}(1)$ is an injective and compact operator for all $\omega \in \Omega$ then there exists an Oseledets decomposition.
\par\smallskip
The papers~\cite{FrLlQu} and~\cite{GTQu} give a proof of the Oseledets decomposition for semi-invertible discrete ergodic transformations. We want to state a version of these results valid for continuous linear skew-product semidynamical systems. We omit the details of the proof that follow standard arguments from the theory mentioned above. The conclusions of this theorem will be relevant in the applications to the theory of delay differential equations.
\begin{theorem}
\label{thm-Oseledets}
Assume $\Phi$ is a measurable linear skew-product semidynamical system satisfying the following:
\begin{enumerate}
\item[\textup{(a)}]
$\Omega$ is a Lebesgue space,
\item[\textup{(b)}]
$X$ is a separable Banach space, and
\item[\textup{(c)}]
$U_{\omega}(1)$ is a compact operator for all $\omega \in \Omega$.
\end{enumerate}
Then $\Phi$ admits an Oseledets decomposition.
\end{theorem}
\begin{proof}[Indication of proof]
The existence of a discrete-time Oseledets decomposition for systems satisfying (a), (b) and (c) was proved in~\cite[Theorem.~A]{GTQu}.  To pass to the continuous-time decomposition we proceed along the lines of the proof of \cite[Theorem.~3.3]{Lian-Lu}.
\end{proof}
\begin{nota} Analogously, from a discrete-time result in~\cite[Corollary~17]{GTQu2}, the existence of an~Oseledets decomposition for $\Phi$ is obtained when $X$ is a separable and reflexive Banach space and (a) is not assumed. The importance of this approach is its constructive nature. More precisely, a way of approximating the Oseledets splitting is provided, which is important in applications.
\end{nota}
\section{Semiflows generated by linear random delay differential equations}\label{sect:semiflows}
This section is devoted to show the applications of the previous theory to random dynamical systems generated by systems of linear random delay differential equations of the form
\begin{equation}\label{main-delay-eq}
z'(t) = A(\theta_{t}\omega) \,z(t) + B(\theta_{t}\omega) \,z(t-1),
\end{equation}
where $z(t) \in \R^N$, $N \ge 2$,  $A(\omega)$, $B(\omega)$ are  $N\times N$ real matrices:
\begin{equation*}
A(\omega) =
\left(\begin{smallmatrix}
a_{11}(\omega)&a_{12}(\omega)&\cdots&a_{1N}(\omega) \\
a_{21}(\omega)&a_{22}(\omega)&\cdots&a_{2N}(\omega) \\
\vdots & \vdots & \ddots & \vdots \\
a_{N1}(\omega)&a_{N2}(\omega)&\cdots&a_{NN}(\omega)
\end{smallmatrix}\right)\,,
\quad
B(\omega) =
\left(\begin{smallmatrix}
b_{11}(\omega)&b_{12}(\omega)&\cdots&b_{1N}(\omega) \\
b_{21}(\omega)&b_{22}(\omega)&\cdots&b_{2N}(\omega) \\
\vdots & \vdots & \ddots & \vdots \\
b_{N1}(\omega)&b_{N2}(\omega)&\cdots&b_{NN}(\omega)
\end{smallmatrix}\right)\,,
\end{equation*}
and $(\OFP,(\theta_t)_{t \in \R})$ is an ergodic metric dynamical system, with $\PP$ complete.\par\smallskip
From now on, the Euclidean norm on $\R^N$ will be denote by $\norm{\cdot}$, $\R^{N \times N}$ will stand for the algebra of $N \times N$ real matrices with the operator or matricial norm induced by the Euclidean norm, i.e., $\norm{A} := \sup\{\norm{A\,u}\mid \norm{u} \le 1\}$, for any $A \in \R^{N \times N}$.
\par\smallskip
We denote by $C$ the Banach space $C([-1, 0], \R^N)$ of continuous $\R^N$-valued functions defined on $[-1,0]$, with the supremum norm (denoted by $\normC{\cdot}$).
\par\smallskip
For $1<p<\infty$, let $L=\R^N\times L_p([-1,0],\R^N)$ be the separable Banach space with the norm
\begin{equation}\label{normL}
\normL{u}=\norm{u_1}+\norm{u_2}_p=\norm{u_1}+\left(\int_{-1}^0 \norm{u_2(s)}^p\,ds\right)^{\!\!1/p}
\end{equation}
for any $u=(u_1,u_2)$ with $u_1\in\R^N$ and $u_2\in L_p([-1,0],\R)$.\par\smallskip
We denote by $J$ the linear mapping from $C$ to $L$
\begin{equation}\label{defiJ}
\begin{array}{cccc}
J\colon & C & \longrightarrow & L\\
&  u &  \mapsto & (u(0),u) \,,
\end{array}
\end{equation}
which belongs to $\mathcal{L}(C,L)$ and $\norm{J}=2$.
In the following, $p$ will be fixed and  $q \in (0, \infty)$ is such that $1/p + 1/q = 1$.
\par\smallskip
Now we introduce the assumptions on the coefficients of the family~\eqref{main-delay-eq}:
\begin{enumerate}[label=(\textbf{S\arabic*}),leftmargin=24pt]
\item\label{S1} (Measurability) $A, B \colon \Omega \to \R^{N \times N}$ are $(\mathfrak{F}, \mathfrak{B}(\R^{N \times N}))$\nobreakdash-\hspace{0pt}measurable.\smallskip
\item\label{S2} (Summability) The $(\mathfrak{F}, \mathfrak{B}(\R))$-measurable functions $a$, $b\colon \Omega \to \R$ defined as $a(\w):=\norm{A(\w)}$ and $b(\w):=\norm{B(\w)}$ have the properties:
\begin{align*}
& \bigl[ \, \Omega \ni \omega \mapsto a(\w) \in \R \, \bigr] \in L_1\OFP, \text{ and}
\\
& \Bigl[ \, \Omega \ni \omega \mapsto \lnplus{\int_{0}^{1}\lvert b(\theta_{r}\w) \rvert^{q} \, dr} \in \R \, \Bigr] \in L_1\OFP.
\end{align*}
\end{enumerate}
\begin{nota}\label{rm:S1}
The following is sufficient for the fulfillment of the second condition in~\ref{S2}:
\begin{equation*}
\bigl[ \, \Omega \ni \omega \mapsto b(\w) \in \R \, \bigr] \in L_q \OFP.
\end{equation*}
Indeed, since $\lvert b \rvert^{q} \in L_1\OFP$ and the measure $\PP$ is invariant, for any $t\in\R$
\begin{equation*}
\int_\Omega \lvert b(\theta_t\omega')\rvert^{q} \,d\PP(\omega') = \int_\Omega \lvert b(\omega')\rvert^{q} \,d\PP(\omega') \,
\end{equation*}
and an application of Fubini's theorem gives that the map
\begin{equation*}
\Bigl[ \, \Omega \ni \omega \mapsto \int_{0}^{1}\lvert b(\theta_{r}\w) \rvert^{q} \, dr \in \R \, \Bigr]
\end{equation*}
belongs to $L_1\OFP$, from which the required statement follows immediately.
\end{nota}
\subsection{Linear skew-product semiflows on $C$ and $L$}\label{subsect:semiflow-C-L}
Before proceeding to the existence of solutions, notice that the coefficients $A$ and $B$ are defined only $\PP$-a.e. on $\Omega$, whereas the theory of Lyapunov exponents requires us to have the solution operator defined on the whole of $\Omega$.  In~particular, we need to have
\begin{align}\label{aL1loc}
\bigl[\, \R\ni t\mapsto a(\theta_t\w)\in \R\,\bigr]&\in L_{1,\text{loc}}(\R)\,,\\
\label{bLqloc}
\bigl[ \, \R\ni t\mapsto b(\theta_t\w)\in \R\,\bigr]& \in L_{q,\text{loc}}(\R)\subset L_{1,\text{loc}}(\R) \,,
\end{align}
for each $\omega \in \Omega$.
Since $a\in L_1\OFP$ and the measure $\PP$ is invariant, for any $t\in\R$
\begin{equation*}
\int_\Omega a(\theta_t{\omega'})\,d\PP(\omega') = \int_\Omega a(\omega')\,d\PP(\omega') \,
\end{equation*}
and an application of Fubini's theorem gives~\eqref{aL1loc} for $\w\in\Omega_1\subset \Omega$,  invariant set of full measure.
Analogously, from the second condition in~\ref{S2}, there is an invariant set $\Omega_2 \subset \Omega$ of full measure such that
$
\bigl[ \R\ni t\mapsto b(\theta_t\w)\in \R\bigr]\in L_{q,\text{loc}}(\R)\subset L_{1,\text{loc}}(\R)\,.
$
Then we can put the value of $A(\w)$  and $B(\w)$ for $\w\in\Omega\setminus (\Omega_1\cap \Omega_2)$ to be equal to zero to obtain~\eqref{aL1loc} and~\eqref{bLqloc} for all $\w\in\Omega$, as needed.
\par\smallskip
As a consequence, for a fixed $\w\in\Omega$ we will denote by  $U_{\omega}^{0}(\cdot)$  the fundamental matrix solution of the system of Carath\'{e}odory linear ordinary  differential equations $z' = A(\theta_{t}\omega)\,z$ and define
\begin{equation}\label{deficd}
c(\omega) := \sup\limits_{0 \le t_1 < t_2 \le 1} \norm{U^{0}_{\theta_{t_1}\omega}(t_2 - t_1)}\,,\quad d(\w):=\biggl(\int_{-1}^0 \!b^q(\theta_{s+1}\w)\,ds\biggr)^{\!\!1/q}\,.
\end{equation}
Notice that $c(\w)\geq 1$. The following two lemmas will be used later.
\begin{lemma}
\label{lm:estimate_U0}
Assume \textup{\ref{S1}} and \textup{\ref{S2}}.  Then for any $\omega \in \Omega$,
\begin{equation*}
c(\omega) \le \exp\biggl( \int_{0}^{1} a(\theta_{\tau}\omega) \, d\tau  \biggr).
\end{equation*}
\end{lemma}
\begin{proof}
Fix $\omega \in \Omega$, $u_0 \in \R^N$, $t_1\geq 0$ and  denote $z(t) := U^0_{\theta_{t_1}\w}(t)\, u_0$.  Since for $t\geq 0$
\begin{equation*}
\norm{z(t)} \le \norm{u_0} + \int_{0}^{t} \norm{A(\theta_{t_1+\tau}\omega)} \,\norm{z(\tau)} \, d\tau = \norm{u_0} + \int_{0}^{t} a(\theta_{t_1+\tau}\omega) \,\norm{z(\tau)} \, d\tau\,,
\end{equation*}
 Gronwall inequality provides
\begin{equation*}
\norm{z(t)}  \le \norm{u_0} \exp\biggl( \int_{0}^{t} a(\theta_{t_1+\tau}\omega) \, d\tau \biggr)
= \norm{u_0} \exp\biggl( \int_{t_1}^{t+t_1} a(\theta_{\tau}\omega) \, d\tau \biggr)
\end{equation*}
for $t\geq 0$. Hence, for $0\leq t_1<t_2\leq 1$
\begin{equation*}
\norm{U^0_{\theta_{t_1}\w}(t_2-t_1)} \leq \exp\biggl( \int_{t_1}^{t_2} a(\theta_{\tau}\omega) \, d\tau \biggr)\le \exp\biggl( \int_{0}^{1} a(\theta_{\tau}\omega) \, d\tau  \biggr)\,,
\end{equation*}
which finishes the proof.
\end{proof}
\begin{lemma}
\label{lm:estimate_pointwise}
 Under assumptions \textup{\ref{S1}} and \textup{\ref{S2}},  for any $\omega \in \Omega$,  $0\leq t \leq 1$, and  $u=(u_1,u_2)\in L=\R^N\times L_p([-1, 0],\R^N)$ there holds
\begin{equation*}
\bigg\|\, U_{\omega}^{0}(t)\,u_1 + \int_{0}^{t} U_{\theta_{\tau}\omega}^{0}(t - \tau) B(\theta_{\tau}\omega)\,u_2(\tau-1) \, d\tau\,\bigg\|
\le c(\w)\, (1+d(\omega)) \, \normL{u}.
\end{equation*}
\end{lemma}
\begin{proof} For simplicity, let us denote the left-hand side by $\norm{z(t)}$. Then, from~\eqref{bLqloc}, \eqref{deficd},  $u_2\in L_p([-1,0],\R^N)$ (i.e. $\norm{u_2}\in L_p([-1,0],\R))$ and H\"{o}lder inequality, we deduce that
\begin{align*}
\norm{z(t)} & \le \norm{U_{\omega}^{0}(t)\,u_1} + \bigg\|\, \int_{0}^{t} U_{\theta_{\tau}\omega}^{0}(t - \tau)\, B(\theta_{\tau}\omega) \,u_2(\tau-1) \, d\tau \bigg\|\\
 & \leq c(\w)\left[ \norm{u_1}+ \int_0^1 b(\theta_{\tau}\omega) \,\norm{u_2(\tau-1)}\,d\tau\right]\\
 & = c(\w)\left[ \norm{u_1}+ \int_{-1}^{t-1} b(\theta_{\tau+1}\omega) \,\norm{u_2(\tau)}\,d\tau\right]\\
 & \leq c(\w)\left[ \norm{u_1}+ \int_{-1}^{0} b(\theta_{\tau+1}\omega) \,\norm{u_2(\tau)}\,d\tau\right]\leq c(\w)\left[\, \norm{u_1}+ d(\w)\,\norm{u_2}_p\,\right]\,,
\end{align*}
which together with~\eqref{normL} finishes the proof.
\end{proof}
\subsubsection{Semiflows on $C([-1,0],\R^N)$}
\label{subsubsect:semiflow-C}
We start with the problem of the existence of solution to the initial value problem
\begin{equation}
\label{eq:IVP-C}
\begin{cases}
z'(t) = A(\theta_{t}\omega) \,z(t) + B(\theta_{t}\omega) \,z(t-1), & t \in [0, \infty)
\\
z(t) = u(t), & t \in [-1, 0],
\end{cases}
\end{equation}
where the initial datum $u$ is assumed to belong to $C=C([-1,0],\R^N)$ and assumptions~\ref{S1} and~\ref{S2} hold.
\par
To emphasize the dependence of the equation (resp.\ the initial value problem) on $\omega \in \Omega$ we will write \eqref{main-delay-eq}$_{\omega}$ (resp.\ \eqref{eq:IVP-C}$_{\omega}$).
\par\smallskip
From~\eqref{aL1loc} and~\eqref{bLqloc}, for a fixed $\w\in\Omega$ and $0\leq t\leq 1$, as shown by Coddington and Levinson~\cite[Theorem~2.1]{book:CL}, the system~\eqref{eq:IVP-C}$_{\omega}$ of Carath\'{e}odory type has a unique solution, which can be written as
\begin{equation}\label{0<t<1}
z(t,\w,u) = U_{\omega}^{0}(t)\, u(0) + \int_0^t U_{\theta_{\tau}\omega}^{0}(t - \tau)\, B(\theta_{\tau}\omega)\, u(\tau-1) \, d\tau\,,
\end{equation}
and, for $1\leq t\leq 2$ as
\begin{equation}\label{1<t<2}
z(t,\w,u)=U_{\omega}^{0}(t-1)\, z(1,\w,u) + \int_1^t U_{\theta_{\tau}\omega}^{0}(t - \tau)\, B(\theta_{\tau}\omega)\, z(\tau-1,\w,u) \, d\tau\,,
\end{equation}
where, as before,  $U_{\omega}^{0}(\cdot)$  denotes the fundamental matrix solution of $z' = A(\theta_{t}\omega)\,z$.
In a recursive way we obtain the formula for $z(t,\w,u)$ for any $t\in[-1,\infty)$.
Moreover, it can be checked that for each $t$ and $r\geq 0$
\begin{equation}\label{cocycle}
z(t+r,\w,u)=z(t,\theta_r\w,z_r(\w,u))\,,
\end{equation}
where $z_r(\w,u)\colon [-1,0]\to \R$, $s\mapsto z(r+s,\w,u)$, and $z_t(\w,u)\in C$ for each $t\geq 0$, $\w\in\Omega$ and $u\in C$.
Therefore, we can define the linear~operator
\begin{equation}\label{defiskpd}
\begin{array}{lccc}
 U^{(C)}_\w(t)\colon & C &\longrightarrow & C \\[.1cm]
                        & u & \mapsto & z_t(\w,u)\,.
\end{array}
\end{equation}
\begin{proposition}\label{Uwbounded} Under assumptions \textup{\ref{S1}} and \textup{\ref{S2}}, $U^{(C)}_\w(t)$ satisfies~\eqref{eq-identity},~\eqref{eq-cocycle} and $U^{(C)}_\w(t)\in \mathcal{L}(C) $ for each $t\geq 0 $ and $\w\in\Omega$.
\end{proposition}
\begin{proof}
Relation~\eqref{eq-identity} is immediate and~\eqref{eq-cocycle} follows from~\eqref{cocycle}. Once that this cocycle property is shown, to prove that $U^{(C)}_\w(t)\in \mathcal{L}(C)$ for $t\geq 0$, it is enough to check that $U^{(C)}_\w(t)$ is a bounded operator for $t\in[0,1]$ and $\w\in\Omega$, which is a consequence of equation~\eqref{0<t<1} and Lemma~\ref{lm:estimate_pointwise} applied to $u_1=u(0)$ and $u_2=u$ because
\begin{align*}
\normC{U^{(C)}_\w(t)\,u}& =\sup_{s\in[-1,0]} \norm{z(t+s,\w,u)}
\leq  \sup_{s\in[-t,0]}\{\normC{u}, \norm{z(t+s,\w,u)}\}\\ &\leq \max\left\{\,\normC{u}, \, c(\w)\,(1+d(\w))\,\left[\, \norm{u(0)}+\,\norm{u}_p\,\right]\,\right\}\\
& \leq 3\,c(\w)\,(1+d(\w))\,\normC{u}\,,
\end{align*}
that is, $\norm{U^{(C)}_\w(t)}\leq 3\,c(\w)\,(1+d(\w))$ for $t\in[0,1]$, which finishes the proof.
\end{proof}
\subsubsection{Semiflows on $L$}
\label{subsubsect:semiflow-L}
Consider the initial value problem
\begin{equation}
\label{eq:IVP}
\begin{cases}
z'(t) = A(\theta_{t}\omega) \,z(t) + B(\theta_{t}\omega)\, z(t-1), & t \in [0, \infty)
\\
z(t) = u_2(t), & t \in [-1, 0),
\\
z(0) = u_1,
\end{cases}
\end{equation}
with initial datum $u=(u_1,u_2)$ belonging to $L=\R^N\times L_p([-1,0],\R^N)$, and assume that~\ref{S1} and~\ref{S2} hold.
\par\smallskip
As in Subsection~\ref{subsubsect:semiflow-C}, from~\eqref{aL1loc} and~\eqref{bLqloc}, for a fixed $\w\in\Omega$ and $0\leq t\leq 1$, as shown in~\cite[Theorem~2.1]{book:CL}, the system~\eqref{eq:IVP}$_{\omega}$ of Carath\'{e}odory type has a unique solution, which can be written as
\begin{equation}\label{eq:L-0<t<1}
z(t,\w,u) = U_{\omega}^{0}(t)\, u_1 + \int_0^t U_{\theta_{\tau}\omega}^{0}(t - \tau)\, B(\theta_{\tau}\omega)\, u_2(\tau-1) \, d\tau\,,
\end{equation} and, as before, in a recursive way the solution of \eqref{eq:IVP}$_{\omega}$ can be written as
\begin{equation}
\label{eq:delay-integral-equation-L}
z(t,\w,u) = U_{\omega}^{0}(t)\, u_1 + \int_{0}^{t} U_{\theta_{\tau}\omega}^{0}(t - \tau) B(\theta_{\tau}\omega) z(\tau-1,\w,u) \, d\tau, \quad t > 0.
\end{equation}
\begin{nota}\label{zuJu}
For $u\in C$ the solution $z(\cdot, \omega, u)$ of~\eqref{eq:IVP-C}$_\w$ coincides with the solution $z(\cdot, \omega, Ju)$  of~\eqref{eq:IVP}$_\w$ where $J$ is defined in~\eqref{defiJ}.
\end{nota}
Moreover, it can be checked that for each $t$ and $r\geq 0$
\begin{equation}\label{5.cocycle}
z(t+r,\w,u)=z(t,\theta_r\w,(z(r,\w,u),z_r(\w,u)))\,,
\end{equation}
where, as before,  $z_t(\w,u)\colon [-1,0]\to \R$, $s\mapsto z(t+s,\w,u)$, which together with Lemma~\ref{lm:estimate_pointwise} show that $z_t(\w,u)\in L_p([-1,0],\R)$ for each $t\geq 0$
and we can define the linear~operator
\begin{equation}\label{5.defiskpd}
\begin{array}{lccc}
 U^{(L)}_\w(t)\colon & L &\longrightarrow & L \\[.1cm]
                        & u & \mapsto & (z(t,\w,u),z_t(\w,u))\,.
\end{array}
\end{equation}
\begin{proposition}\label{5.Uwbounded} Under assumptions~\textup{\ref{S1}} and \textup{\ref{S2}}, $U^{(L)}_\w(t)$ satisfies~\eqref{eq-identity},~\eqref{eq-cocycle} and $U^{(L)}_\w(t)\in \mathcal{L}(L) $ for each $t\geq 0 $ and $\w\in\Omega$.
\end{proposition}
\begin{proof} Relation~\eqref{eq-identity} is immediate and~\eqref{eq-cocycle} follows from~\eqref{5.cocycle}. Again, once that this cocycle property is shown, to prove that $U^{(L)}_\w(t)\in \mathcal{L}(L)$ for $t\geq 0$, it is enough to check that $U^{(L)}_\w(t)$ is a bounded operator for $t\in[0,1]$ and $\w\in\Omega$, which is a consequence of~equation~\eqref{eq:L-0<t<1} and Lemma~\ref{lm:estimate_pointwise} because
\begin{align*}
\normL{U^{(L)}_\w(t)\,u}&=\norm{z(t,\w,u)}+\biggl(\int_{-1}^0 \norm{z(t+s,\w,u)}^p\,ds\biggr)^{\!\!1/p}\leq c(\w)\,(1+d(\w))\,\normL{u}\\
&\quad +\biggl(\int_{-1}^{-t} \norm{u_2(t+s)}^p\,ds\biggr)^{\!\!1/p}+\biggl(\int_{-t}^0 \norm{z(t+s,\w,u)}^p\,ds\biggr)^{\!\!1/p}\\
&\leq 3\,c(\w)\,(1+d(\w))\,\normL{u}\,,
\end{align*}
that is, $\norm{U^{(L)}_\w(t)}\leq 3\,c(\w)\,(1+d(\w))$ for $t\in[0,1]$, which finishes the proof.
\end{proof}
\subsubsection{Connections between semidynamical systems on $C$ and on $L$}
First of all, observe that for any $t \ge 0$ and any $\omega \in \Omega$ one has
\begin{equation}
\label{eq:L-C-1}
U^{(L)}_{\omega}(t) \circ J = J \circ U^{(C)}_{\omega}(t)\,,
\end{equation}
where $J$ is defined in~\eqref{defiJ}.
We can also define for $t\geq 1$ the linear operator
\begin{equation}\label{5.defiskpdLC}
\begin{array}{lccc}
 U^{(L,C)}_\w(t)\colon & L &\longrightarrow & C \\[.1cm]
                        & u & \mapsto & z_t(\w,u)\,.
\end{array}
\end{equation}
\begin{proposition} Under assumptions \textup{\ref{S1}} and \textup{\ref{S2}},  for each $t\geq 1$ and $\w\in\Omega$,  $U^{(L,C)}_\w(t)\in \mathcal{L}(L, C)$ and  is a compact operator.
\end{proposition}
\begin{proof}
 First we check the result for $t=1$ and each $\w\in\Omega$.  Take $u\in L$; from Lemma~\ref{lm:estimate_pointwise} we deduce that
\begin{equation}\label{normU)L,C)(1)}
\normC{U^{(L,C)}_\w(1)\,u}=\sup_{s\in[-1,0]} \norm{z(1+s,\w,u)}\leq c(\w)(1+d(\w))\,\normL{u}\,,
\end{equation}
 which proves that $U^{(L,C)}_\w(1)\in \mathcal{L}(L, C)$, as claimed.   For the compactness, fix $\w\in\Omega$, $0\leq s_1\leq s_2\leq 0$, $u=(u_1,u_2)\in L$ and note that
\begin{equation*}
(U^{(L,C)}_{\w}(1)\,u)(s_2)-(U^{(L,C)}_{\w}(1)\,u)(s_1)=z(1+s_2,\w,u)-z(1+s_1,\w,u)\,.
\end{equation*}
Then denoting $t_1=1+s_1$ and $t_2=1+s_2$, we deduce that
\begin{align*}
\norm{z(t_2,\w,u)-z(t_1,\w,u)}& \leq \int_{t_1}^{t_2} a(\theta_s\omega)\,\norm{z(s,\w,u)}\,ds+\int_{t_1}^{t_2} b(\theta_s\omega)\,\norm{u_2(s-1)}\,ds\,\\
& = \int_{t_1}^{t_2} a(\theta_s\omega)\,\norm{z(s,\w,u)}\,ds +\int_{s1}^{s_2}b(\theta_{s+1}\w)\,\norm{u_2(s)}\,ds\,.
\end{align*}
Moreover, again Lemma~\ref{lm:estimate_pointwise}, \eqref{bLqloc} and H\"{o}lder inequality provide
\begin{align*}
\norm{z(t_2,\w,u)-z(t_1,\w,u)}&\leq c(\w)\,(1+d(\w))\,\norm{u}_L \int_{t_1}^{t_2} a(\theta_s\omega)\,ds\\
&\quad + \norm{u_2}_{p}\,\biggl(\int_{s_1}^{s_2} \!b^q(\theta_{s+1}\w)\,ds\biggr)^{\!\!1/q}\,ds\,,
\end{align*}
which together with \eqref{aL1loc},~\eqref{bLqloc} and $\norm{u_2}_p\leq \normL{u}$ imply the equicontinuity of the set  $\left\{\, U^{(L,C)}_{\omega}(1)\,u \mid  \normL{u} \le k \, \right\}$, and by Ascoli--Arzel\`a theorem, the precompactness, as needed.\par\smallskip
Finally, for $t\geq 1$, from $z_{t}(\w,u)=z_{t-1}(\theta_{1}\w,z_{1}(\w,u))$ we deduce that
\begin{equation}\label{eq:L-C-L}
U^{(L,C)}_\w(t)= U_{\theta_{1}\w}^{(C)}(t-1)\circ U_{\w}^{(L,C)}(1)\,,
\end{equation}
which shows, from Proposition~\ref{Uwbounded}, that for $t\geq 1$, $U^{(L,C)}_\w(t)$ is the composition of a compact operator and a linear operator. Hence, $U^{(L,C)}_\w(t)\in \mathcal{L}(L, C)$ and it is a compact operator for $t\geq 1$ and each $\w\in\Omega$, as stated.
\end{proof}
\begin{corollary}
Under assumptions \textup{\ref{S1}} and \textup{\ref{S2}},  for any $\omega \in \Omega$ and $t \ge 1$, $U^{(L)}_{\omega}(t)$ and $U^{(C)}_\omega(t)$ are compact operators.
\end{corollary}
\begin{proof}
Since $J$ is a bounded operator from $C$ to $L$, the result is a consequence of the previous proposition and the relations
\begin{equation}
\label{eq:L-C-2}
U^{(L)}_{\omega}(t) = J \circ U^{(L, C)}_{\omega}(t), \quad U^{(C)}_{\omega}(t) = U^{(L, C)}_{\omega}(t) \circ J
\end{equation}
for any $t \ge 1$ and any $\omega \in \Omega$.
\end{proof}
\begin{lemma}
\label{lm:estimate_L-to-C}
Assume \textup{\ref{S1}} and \textup{\ref{S2}}.  Then
\begin{enumerate}
\item[\textup{(i)}]
for any $\omega \in \Omega$, any $u \in L$  and any $t \ge 0$ there holds
\begin{equation*}
\normC{U_{\omega}^{(L, C)}(t + 1)\, u} \le c(\theta_t\w)\,(1+d(\theta_t\w))\, \normL{U_{\omega}^{(L)}(t) \,u}\,,
\end{equation*}
\item[\textup{(ii)}]
for any $\omega \in \Omega$, any $u\in L$  and any $t \ge 0$ there holds
\begin{equation*}
\normL{U_{\omega}^{(L)}(t + 1)\, u} \le 2\, c(\theta_t\w)\,(1+d(\theta_t\w))\, \normL{U_{\omega}^{(L)}(t) \,u}\,.
\end{equation*}
\end{enumerate}
\end{lemma}
\begin{proof}
First, from~\eqref{5.cocycle} we deduce that  $U_\w^{(L,C)}(t+1)=U_{\theta_t\w}^{(L,C)}(1)\circ U_\w^{(L)}(t)$ for each $t\geq 0$,  which together with~\eqref{normU)L,C)(1)} proves (i). (ii) is a consequence of (i),~\eqref{eq:L-C-2} and $\norm {J}=2$.
\end{proof}
\subsubsection{Measurability} In order to show that
$\Phi = \allowbreak \big((U_\w^{(L)}(t))_{\w \in \Omega, t \in \R^{+}}, \allowbreak (\theta_t)_{t\in\R}\big)$ and $\Phi = \allowbreak \big((U_\w^{(C)}(t))_{\w \in \Omega, t \in \R^{+}}, \allowbreak (\theta_t)_{t\in\R}\big)$   are  \mlspss\ we start with the following auxiliary lemma.
\begin{lemma}\label{lm:auxiliary-measurable}
 Under \textup{\ref{S1}} and \textup{\ref{S2}},  for $u=(u_1,u_2) \in L=\R^N\times L_p([-1, 0],\R^N)$ and $t \in (0, 1]$ the mapping
\begin{equation*}
\big[\,\Omega \ni \omega \mapsto z(t,\w,u)\in \R^N\,\big] \text{ is }(\mathfrak{F}, \mathfrak{B}(\R^N))\text{-measurable}\,,
\end{equation*}
where $z(t,\w,u)$ denotes the solution of~\eqref{eq:IVP}$_\w$, given by the formula~\eqref{eq:L-0<t<1}.
\end{lemma}
\begin{proof}
It is well known (see, e.g., \cite[Example 2.2.8]{Arn}) that the mapping
\begin{equation*}
\bigl[ \Omega \times [0, t] \ni (\omega, \tau) \mapsto U_{\omega}^{0}(\tau) \in \R^{N \times N}\bigr]\;\; \text{is} \;\; (\mathfrak{F} \otimes \mathfrak{B}([0, t]), \mathfrak{B}(\R^{N \times N}))\text{-measurable}\,.
\end{equation*}
As a consequence, for each $t\in(0,1]$  the mapping
\begin{equation}\label{measU0}
\bigl[\,\Omega \ni \omega \mapsto U_{\omega}^{0}(t)\,u_1\in \R^N \, \bigr]\;\; \text{
is}\;\;  (\mathfrak{F}, \mathfrak{B}(\R^N))\text{-measurable}.
\end{equation}
\smallskip
Another consequence is that the map
$\bigl[\, \Omega\times[0,t] \ni (\omega,\tau) \mapsto U_{\theta_{\tau}\omega}^{0}(t - \tau) \in \R^{N \times N} \, \bigr]$
is $(\mathfrak{F} \otimes \mathfrak{B}([0, t]), \mathfrak{B}(\R^{N \times N}))$\nobreakdash-\hspace{0pt}measurable, and hence, from~\ref{S1} and taking a Borel representation of the function $u_2\in L_p([-1,0],\R^N)$, we deduce that
\begin{equation*}
\big[\,\Omega \times [0,t]\ni (\w,\tau)\mapsto U_{\theta_{\tau}\omega}^{0}(t - \tau)\, B(\theta_{\tau}\omega)\, u_2(\tau-1) \in\R^N\,\big]
\end{equation*}
is $(\mathfrak{F} \otimes \mathfrak{B}([0, t]), \mathfrak{B}(\R^{N}))$\nobreakdash-\hspace{0pt}measurable. Therefore, an application of Fubini's theorem show that for each $t\geq 1$
\begin{equation*}
\Big[\,\Omega \ni \w\mapsto \int_0^t U_{\theta_{\tau}\omega}^{0}(t - \tau)\, B(\theta_{\tau}\omega)\, u_2(\tau-1)\,d\tau \in\R^N\,\Big] \; \text{is}\;  (\mathfrak{F}, \mathfrak{B}(\R^N))\text{-measurable}\,,
\end{equation*}
which together with~\eqref{measU0} and formula \eqref{eq:L-0<t<1} finishes the proof.
\end{proof}
\begin{lemma}
\label{lm:measurable-C}
Assume \textup{\ref{S1}} and \textup{\ref{S2}}.  For $u\in C$ and $t > 0$ fixed, the mapping  $\bigl[ \, \Omega \ni \omega \mapsto U^{(C)}_{\omega}(t)\,u \in C \, \bigr]$ is $(\mathfrak{F}, \mathfrak{B}(C))$-measurable.
\end{lemma}
\begin{proof}
Since $U_\w^{(C)}$ satisfies the cocycle property~\eqref{eq-cocycle}, it suffices to prove the result for $t \in (0, 1]$. By a variant of Pettis' Measurability Theorem (see, e.g., \cite[Corollary 4 on pp. 42--43]{DiUhl} with an appropriate norming set for the dual space of $C=C([-1,0],\R^N)$), the mapping
\[\bigl[ \, \Omega \ni \omega \mapsto U^{(C)}_{\omega}(t)\,u \in C \, \bigr] \text{ is } (\mathfrak{F}, \mathfrak{B}(C))\text{-measurable}
\]
 if and only if
 \[\bigl[ \, \Omega \ni \omega \mapsto (U^{(C)}_{\omega}(t)\,u)(\tau) \in \R^N \, \bigr] \text{ is } (\mathfrak{F}, \mathfrak{B}(\R^N))\text{-measurable}
\]
for any $\tau\in[-1,0]$, a consequence of Lemma~\ref{lm:auxiliary-measurable} applied to $Ju=(u(0),u)\in L$ because $(U^{(C)}_{\omega}(t)\,u)(\tau)=z(t+\tau,\w,u)=z(t+\tau,\w, Ju)$ (see Remark~\ref{zuJu}).
\end{proof}
\begin{proposition}\label{prop:delay-semiflow-C}
Under \textup{\ref{S1}} and \textup{\ref{S2}}, $\bigl((U^{(C)}_{\omega}(t))_{\omega \in \Omega, t \in [0,\infty)}, (\theta_{t})_{t \in \R}\bigr)$ is a measurable linear skew\nobreakdash-\hspace{0pt}product semiflow on $C$ covering $\theta$.
\end{proposition}
\begin{proof}
Since for $\w\in\Omega$ and $u\in C$ fixed the mapping $\bigl[\,\R^+\ni t\mapsto U_\w^{(C)}(t)\,u\in C\,\bigr]$ is easily seen to be continuous, the fact that the mapping
\[
\bigl[\,\R^+\times \Omega\times L\ni(t,\w,u) \mapsto U_\w(t)\,u\in C\,\bigr]\]
is $(\mathfrak{B}(\R^+)\otimes\mathfrak{F}\otimes\mathfrak{B}(C),
\mathfrak{B}(C))$-measurable
follows from Proposition~\ref{Uwbounded}, Lemma~\ref{lm:measurable-C} and~\cite[Lemma 4.51 on pp.~153]{AliB}. The rest of the properties have been already checked, so that $\Phi$ is a \mlsps, as claimed.
\end{proof}
\begin{lemma}
\label{lm:measurable-L}
Assume \textup{\ref{S1}} and \textup{\ref{S2}}.  For $u\in L$ and $t > 0$ fixed, the mapping  $\bigl[ \, \Omega \ni \omega \mapsto U^{(L)}_{\omega}(t)\,u\in L \, \bigr]$ is $(\mathfrak{F}, \mathfrak{B}(L))$-measurable.
\end{lemma}
\begin{proof}
It follows from~\eqref{eq-cocycle} that it suffices to prove the result for $t\in(0,1]$ only.  Since $L$ is separable, from Pettis' Theorem (see~Hille and Phillips~\cite[Theorem 3.5.3 and Corollary 2 on pp. 72--73]{HiPhi}) the weak and strong measurability notions are equivalent and therefore,  it is enough to check that for each $u^*\in L^*$ the mapping
\begin{equation}\label{5.u*measur}
\bigl[\,\Omega\ni\w\mapsto \langle u^*, U_\w^{(L)}(t)\,u\rangle \in \R\,\bigr] \text{ is } (\mathfrak{F},
\mathfrak{B}(\R))\text{-measurable}\,.
\end{equation}
Fixing $u^*=(u_1^*,u_2^*)\in\R^N\times L_q([-1,0],\R^N)$, $u=(u_1,u_2)\in \R^N\times L_p([-1,0],\R^N)$ and $t\in(0,1]$, we have
\begin{align*}
\langle u^*, U_\w^{(L)}(t)\,u\rangle & = (u_1^*)^t z(t,\w,u)+ \int_{-1}^0 (u_2^*(\tau))^t z(t+\tau,\w,u)\,d\tau\,,\\
\int_{-1}^0\!\! (u_2^*(\tau))^t z(t+\tau,\w,u)\,d\tau
 &= \int_{-1}^{-t}\!\!\! (u_2^*(\tau))^t u_2(t+\tau)\,ds
 +\int_{0}^t \!\!(u_2^*(\tau-1))^t z(\tau,\w,u)\,d\tau\,.
\end{align*}
 Hence, Lemma~\ref{lm:auxiliary-measurable} and similar arguments prove that~\eqref{5.u*measur} holds, as stated.
\end{proof}
\begin{proposition}
\label{prop:delay-semiflow-L}
Under \textup{\ref{S1}} and \textup{\ref{S2}}, $\bigl((U^{(L)}_{\omega}(t))_{\omega \in \Omega, t \in [0,\infty)}, (\theta_{t})_{t \in \R}\bigr)$ is a measurable linear skew\nobreakdash-\hspace{0pt}product semiflow on $L$ covering $\theta$.
\end{proposition}
\begin{proof}
Since for $\omega \in \Omega$ and $u\in L$ the mapping $\bigl[\, \R^+\ni t \mapsto U^{(L)}_{\omega}(t)\,u \in L \, \bigr]$ is continuous, as in Proposition~\ref{prop:delay-semiflow-C}  the result follows from  Proposition~\ref{5.Uwbounded}, Lemma~\ref{lm:measurable-L} and~\cite[Lemma 4.51 on pp.~153]{AliB}.
\end{proof}
\begin{nota}
 It can also be proved that for $u \in L$ and $t \ge 1$ fixed, the mapping  $\bigl[ \, \Omega \ni \omega \mapsto U^{(L,C)}_{\omega}(t)\, u \in C \, \bigr]$ is $(\mathfrak{F}, \mathfrak{B}(C))$-measurable.
\end{nota}
\section{Lyapunov exponents}\label{sect:Lyapunov}
In this section it is shown that the Lyapunov exponents for the two cases considered in the previous section are the same, and that the Oseledets decomposition are related by natural embeddings.
\par\smallskip
Throughout the whole section we assume that~\ref{S1} and~\ref{S2} holds.
\begin{lemma}\label{lm:zero}
Consider the functions $c$ and $d$ defined on~\eqref{deficd}. There exists an invariant subset $\Omega_1 \subset \Omega$ with $\PP(\Omega_1) = 1$ such that for any $\omega \in \Omega_1$
\begin{align*}
\limsup\limits_{t \to \infty} \frac{\ln c(\theta_{t}\omega)}{t} \le 0\,, &\quad \limsup\limits_{t \to \infty} \frac{\ln d(\theta_{t}\omega)}{t} \le 0,\\ \liminf_{t\to -\infty}\;\frac{\ln c(\theta_{t}\omega)}{t}\geq 0\,, & \quad \liminf_{t\to -\infty}\;\frac{\ln d(\theta_{t}\omega)}{t}\geq 0\,.
\end{align*}
\end{lemma}
\begin{proof}
From~\ref{S2} and Lemma~\ref{lm:estimate_U0}, $\lnplus{c} \in L^1\OFP$.
Hence, there exists a set $\tilde{\Omega} \subset \Omega$, $\PP(\tilde{\Omega}) = 1$, such that for any $\omega \in \tilde{\Omega}$
\begin{equation*}
\lim\limits_{n \to \infty} \frac{\lnplus{c(\theta_{n}\omega)}}{n} = 0, \quad \text{ and therefore } \quad
\limsup\limits_{n \to \infty} \frac{\ln{c(\theta_{n}\omega)}}{n} \le 0\,.
\end{equation*}
In addition,  for $t > 0$ it can be checked from the definition of $c$ and the cocycle property~\eqref{eq-cocycle} satisfied by $U^0_\w$ that
$ c(\theta_{t}\omega) \le  c(\theta_{\lfloor t \rfloor + 1}\omega)\, c(\theta_{\lfloor t \rfloor}\omega)$, where $\lfloor t \rfloor$ denotes the integer part of the real number $t$. From this which we conclude that
\begin{equation*}
\limsup\limits_{t \to \infty} \frac{\ln{c(\theta_{t}\omega)}}{t} \le 0
\end{equation*}
for any $\omega \in \tilde{\Omega}$. In order to proof the second inequality it is enough to notice that, again from~\ref{S2}, we deduce that $\lnplus{d} \in L^1\OFP$ and in this case
$ d(\theta_{t}\omega) \le  d(\theta_{\lfloor t \rfloor}\omega)+ d(\theta_{\lfloor t \rfloor + 1}\omega)$, which is an easy consequence of the definition of $d$. The other two inequalities follow from the application of the previous ones to the reversed flow $\widehat\theta_t\omega=\theta_{-t}\omega$.
\end{proof}
The next result shows that we have the same Lyapunov exponents, independent of the choice of the Banach space, $C$ or $L$.
\begin{proposition}
\label{prop:Lyapunov-exponents-equal}
Let $\Omega_1$ be as in \textup{Lemma~\ref{lm:zero}}.
\begin{enumerate}
\item[\textup{(1)}]
Assume that for some $\omega \in \Omega_1$, $u \in C$ and $\lambda \in [-\infty, \infty)$ there holds
\begin{equation*}
\lambda = \lim\limits_{t \to \infty} \frac{\ln{\normC{U^{(C)}_{\omega}(t)\, u}}}{t}\,.\text{ Then }
\lambda = \lim\limits_{t \to \infty} \frac{\ln{\normL{U^{(L)}_{\omega}(t)\, (Ju)}}}{t}.
\end{equation*}
\item[\textup{(2)}]
Assume that for some $\omega \in \Omega_1$, $u \in L$ and $\lambda \in [-\infty, \infty)$ there holds
\begin{equation*}
\lambda = \lim\limits_{t \to \infty} \frac{\ln{\normL{U^{(L)}_{\omega}(t) \,u}}}{t}
\,.\text{ Then }
\lambda = \lim\limits_{t \to \infty} \frac{\ln{\normC{U^{(L, C)}_{\omega}(t) \,u}}}{t}.
\end{equation*}
\end{enumerate}
\end{proposition}
\begin{proof}
Recall that for any real sequences $(a_m)_{m = 1}^{\infty}$, $(b_m)_{m = 1}^{\infty}$ with $\limsup\limits_{m \to \infty}a_m < \infty$ and $\limsup\limits_{m \to \infty}b_m < \infty$ there holds
\begin{align}
\label{eq:liminf1}
\liminf\limits_{m \to \infty} {(a_m + b_m)} & {} \le \limsup\limits_{m \to \infty} a_m + \liminf\limits_{m \to \infty} b_m,
\\
\label{eq:liminf2}
\limsup\limits_{m \to \infty} {(a_m + b_m)} & {} \le \limsup\limits_{m \to \infty} a_m + \limsup\limits_{m \to \infty} b_m\,.
\end{align}
\par
(1) From~\eqref{eq:L-C-2}, Lemma~\ref{lm:estimate_L-to-C}(i), \eqref{eq:liminf1}, Lemma~\ref{lm:zero} and~\eqref{eq:L-C-1}   we deduce the following chain of inequalities
\begin{align*}
\lambda & = \lim_{t \to \infty} \frac{\ln{\normC{U^{(C)}_{\omega}(t+ 1) \,u}}}{t + 1}  
 =\lim\limits_{t \to \infty} \frac{\ln{\normC{U^{(L, C)}_{\omega}(t + 1)\,(Ju)}}}{t}\\
 & \leq  \limsup_{t \to \infty} \frac{\ln{[c(\theta_{t}\omega)\,(1+d(\theta_{t}\omega))]}}{t} + \liminf_{t \to \infty} \frac{\ln{\normL{U^{(L)}_{\omega}(t) \,(J u)}}}{t}\\
& \le \liminf_{t \to \infty} \frac{\ln{\normL{U^{(L)}_{\omega}(t) \,(J u)}}}{t} \le \limsup\limits_{t \to \infty} \frac{\ln{\normL{U^{(L)}_{\omega}(t) \,(J u)}}}{t}\\
 & = \limsup_{t \to \infty} \frac{\ln{\normL{(J \circ U^{(C)}_{\omega}(t))\, u}}}{t} \le \lim\limits_{t \to \infty} \frac{\ln{\normC{U^{(C)}_{\omega}(t)\,u}}}{t} = \lambda\,,
\end{align*}
which finishes the proof of (1).\par\smallskip
(2) Similarly, from~\eqref{eq:L-C-2}, Lemma~\ref{lm:estimate_L-to-C}(i),~\eqref{eq:liminf2} and~Lemma~\ref{lm:zero} we obtain
\begin{align*}
 \lambda & = \lim_{t\to \infty} \frac{\ln{\normL{U^{(L)}_{\omega}(t) \,u}}}{t}
= \lim_{t \to \infty} \frac{\ln{\normL{(J \circ U^{(L, C)}_{\omega}(t+1)) \,u}}}{t+1}\\
&\le \liminf_{t \to \infty} \frac{\ln{\normC{U^{(L, C)}_{\omega}(t+1) \,u}}}{t+1} \le \limsup_{t \to \infty} \frac{\ln{\normC{U^{(L, C)}_{\omega}(t+1) \,u}}}{t+1}\\
& \le \limsup_{t \to \infty}\frac{\ln{[c(\theta_{t}\omega)\,(1+d(\theta_{t}\omega))]}}{t} + \lim_{t \to \infty} \frac{\ln{\normL{U^{(L)}_{\omega}(t) \, u}}}{t+1}\\
& \le \lim_{t \to \infty} \frac{\ln{\normL{U^{(L)}_{\omega}(t)  \,u}}}{t+1}
=  \lim_{t \to \infty} \frac{\ln{\normL{U^{(L)}_{\omega}(t) u}}}{t} = \lambda\,,
\end{align*}
and (2) holds, as stated.
\end{proof}
As a consequence, the top Lyapunov exponents for $U^{(C)}_{\omega}$ and $U^{(L)}_{\omega}$ coincide.
\begin{proposition}\label{prop:top-equal}
Let $\lambda^{(C)}_{\mathrm{top}} \in [-\infty, \infty)$ and $\lambda^{(L)}_{\mathrm{top}} \in [-\infty, \infty)$ denote the top Lyapunov exponents for $U^{(C)}_{\omega}$ and $U^{(L)}_{\omega}$ respectively,
i.e. \begin{equation}\label{for:topLC}
\lim\limits_{t \to \infty} \frac{\ln{\norm{U^{(C)}_{\omega}(t)}}}{t} = \lambda^{(C)}_{\mathrm{top}}, \quad \lim\limits_{t \to \infty} \frac{\ln{\norm{U^{(L)}_{\omega}(t)}}}{t} = \lambda^{(L)}_{\mathrm{top}},
\end{equation}
for $\PP$-a.e.\ $\omega \in \Omega$.
Then $\,\lambda^{(C)}_{\mathrm{top}} = \lambda^{(L)}_{\mathrm{top}}$.
\end{proposition}
\begin{proof}
Let $\omega \in \Omega_1$ satisfying~\eqref{for:topLC}. From~\eqref{eq:L-C-2},  $U_\w^{(L,C)}(t+1)=U_{\theta_t\w}^{(L,C)}(1)\circ U_\w^{(L)}(t)$, \eqref{normU)L,C)(1)} and $\norm {J}=2$  we deduce that
\[\norm{U_\w^{(C)}(t+1)}\leq 2\,\norm{U_\w^{(L,C)}(t+1)}\leq 2\, c(\theta_t\w)\,(1+d(\theta_t\w))\, \norm{U_{\omega}^{(L)}(t)}\,,\]
which together with Lemma~\eqref{lm:zero} provides $\lambda^{(C)}_{\mathrm{top}} \leq
\lambda^{(L)}_{\mathrm{top}}$. \par\smallskip
On the other hand, from~\eqref{eq:L-C-2},~\eqref{eq:L-C-L}, \eqref{normU)L,C)(1)} and $\norm{J}=2$ we obtain
\[ \norm{U_\w^{(L)}(t+1)}\leq 2\, \norm{U_{\theta_{1}\w}^{(C)}(t)}\,\norm{ U_{\w}^{(L,C)}(1)}\leq 2\, c(\w)\,(1+d(\w))\, \norm{U_{\theta_{1}\omega}^{(C)}(t)}\,,\]
which together with the invariance of the set satisfying~\eqref{for:topLC} implies that $\lambda^{(L)}_{\mathrm{top}} \leq
\lambda^{(C)}_{\mathrm{top}}$, and finishes the proof.
\end{proof}
\begin{nota}\label{topLc-equal}
In view of the above we will denote the common value of $\lambda^{(C)}_{\mathrm{top}} = \lambda^{(L)}_{\mathrm{top}}$ by $\lambda_{\mathrm{top}}$.
\end{nota}
From now on, we assume that both $\Phi^{(C)}$ and $\Phi^{(L)}$ admit an Oseledets decomposition.
From Proposition~\ref{prop:limit-everywhere} we can find  a common invariant set $\Omega_0 \subset \Omega_1$ with $\PP(\Omega_0) = 1$, such that for any $\omega \in \Omega_0$ and each nonzero $u \in C$ the limit
\begin{equation*}
\lim_{t \to \infty} \frac{\normC{U^{(C)}_{\omega}(t) \,u}}{t}
\end{equation*}
exist and equals either $-\infty$ or some Lyapunov exponent $\lambda_i^{(C)}$, and for each nonzero $u \in L$ the limit
\begin{equation*}
\lim_{t \to \infty} \frac{\normL{U^{(L)}_{\omega}(t)\, u}}{t}
\end{equation*}
exist and equals either $-\infty$ or some Lyapunov exponent~$\lambda_i^{(L)}$.
\begin{theorem}
\label{thm:Lyapunov-equal}  Assume that both $\Phi^{(C)}$ and $\Phi^{(L)}$  admit an Oseledets decomposition. For any $i \in \{1, \ldots, k\}$ in case \textup{(I)}, or any $i \in \NN$ in case \textup{(II)} there holds:
\begin{enumerate}
\item[\textup{(a)}]
$\lambda^{(C)}_{i} = \lambda^{(L)}_{i}$.
\smallskip
\item[\textup{(b)}]
$F^{(C)}_i(\omega) = J^{-1} (F^{(L)}_i(\omega))$, $\PP$-a.e.~on $\Omega$.
\smallskip
\item[\textup{(c)}]
$U^{(L,C)}_{\omega}(t)\, F^{(L)}_i(\omega) \subset F^{(C)}_i(\theta_{t}\omega))$ for $t \ge 1$, $\PP$-a.e.~on $\Omega$.
\smallskip
\item[\textup{(d)}]
$F^{(C)}_{\infty}(\omega) = J^{-1} (F^{(L)}_{\infty}(\omega))$, $\PP$-a.e.~on $\Omega$.
\smallskip
\item[\textup{(e)}]
$U^{(L,C)}_{\omega}(t)\, F^{(L)}_{\infty}(\omega) \subset F^{(C)}_{\infty}(\theta_{t}\omega))$ for $t \ge 1$, $\PP$-a.e.~on $\Omega$.
\end{enumerate}
\end{theorem}
\begin{proof}
We prove first (a), (b) and (c) by~induction. From Proposition~\ref{prop:top-equal} we have $\lambda_1^{(C)}=\lambda_1^{(L)}=\lambda_{\mathrm{top}}$ so that (a) holds for $i=1$. From the definition of $F_1^{(C)}(\w)$,  if $u\in F_1^{(C)}(\w)$ we deduce that
$\lim\limits_{t \to \infty}  (1/t) \ln{\normC{U_\omega^{(C)}(t)\,u}}< \lambda_1^{(C)}=\lambda_1^{(L)}\,$ and hence, Proposition~\ref{prop:Lyapunov-exponents-equal}(1) provides $\lim\limits_{t \to \infty}  (1/t) \ln{\normL{U_\omega^{(L)}(t)\,(Ju)}}<\lambda_1^{(L)}$, that is, $Ju \in F_1^{(L)}(\w)$.
If $Ju \in F_1^{(L)}(\w)$, then the definition of $F_1^{(L)}(\w)$ gives that
$\lim\limits_{t \to \infty}  (1/t) \ln{\normL{U_\omega^{(L)}(t)\,(Ju)}} < \lambda_1^{(L)} = \lambda_1^{(C)}$, and hence, Proposition~\ref{prop:Lyapunov-exponents-equal}(2) together with~\eqref{eq:L-C-2} imply $\lim\limits_{t \to \infty}  (1/t) \ln{\normC{U_\omega^{(L)}(t)\, u}} < \lambda_1^{(C)}$, that is, $u \in F_1^{(L)}(\w)$. Therefore, statement (b) holds for $i = 1$.
Analogously,  for any $\hat u\in F_1^{(L)}(\w)$ we have that
$\lim\limits_{t \to \infty}  (1/(s+t)) \ln {\normL{U_\omega^{(L)}(s+t)\,\hat u}}< \lambda_1^{(L)}=\lambda_1^{(C)}\,$ and from Proposition~\ref{prop:Lyapunov-exponents-equal}(2) we deduce that
$\lim\limits_{t \to \infty} (1/(s+t)) \ln{\normC{U_\omega^{(L,C)}(s+t)\,\hat u}<\lambda_1^{(C)}}$,
which together with  $U_\omega^{(L,C)}(s+t)= U_{\theta_t\omega}^{(C)}(s)\circ U^{(L,C)}_\omega(t)$ show that
\[ \lim_{s\to\infty} \frac{\ln{\normC{U_{\theta_{t}\omega}^{(C)}(s)\circ U_\omega^{(L,C)}(t)\,\hat u}}}{s} < \lambda_1^{(C)}\,,\]
that is, $U_\omega^{(L,C)}(t)\, u\in F^{(C)}_1(\theta_{t}\omega)$ and (c) holds for $i=1$.
\par\smallskip
Next, let $i>2$ be such that $\lambda_j^{(C)} = \lambda_j^{(L)}$ and $F^{(C)}_j(\omega) = J^{-1} (F^{(L)}_j(\omega))$ for all $\omega \in \Omega_0$ and $j = 1, 2, \dots, i-1$.  For any $\omega \in \Omega_0$ and $u \in F_{i - 1}^{(C)}(\omega) \setminus F_i^{(C)}(\omega)$ there holds $\lim\limits_{t \to \infty} (1/t) \ln{\normC{U^{(C)}_{\omega}(t) \,u}} = \lambda^{(C)}_{i}$ and from Proposition~\ref{prop:Lyapunov-exponents-equal}(1) we deduce that $\lim\limits_{t \to \infty} (1/t) \ln{\normL{U^{(L)}_{\omega}(t) (Ju)}} = \lim\limits_{t \to \infty} (1/t) \ln{\normC{U^{(C)}_{\omega}(t) \,u}} = \lambda^{(C)}_{i}$.  Moreover, by induction hypothesis on (b), $Ju \in F_{i - 1}^{(L)}(\omega)$, hence $\lim\limits_{t \to \infty} (1/t) \ln{\normL{U^{(L)}_{\omega}(t) (Ju)} }= \lambda^{(L)}_{m}$ for some $m \ge i$.  Therefore
\begin{equation}
\label{eq:auxiliary-1}
\lambda^{(C)}_{i} = \lambda^{(L)}_{m} \quad \text{for some } m \ge i.
\end{equation}
For any $\hat{u} \in F_{i - 1}^{(L)}(\omega) \setminus F_{i}^{(L)}(\omega)$ there holds $\lim\limits_{t \to \infty} (1/t) \ln{\normL{U^{(L)}_{\omega}(t)\, \hat{u}}} = \lambda^{(L)}_{i}$.  By Proposition~\ref{prop:Lyapunov-exponents-equal}(2), $\lim\limits_{t \to \infty} (1/t) \ln{\normC{U^{(L, C)}_{\omega}(t) \,\hat{u}}} = \lim\limits_{t \to \infty} (1/t) \ln{\normL{U^{(L)}_{\omega}(t)\, \hat{u}}} = \lambda^{(L)}_{i}$.  On the other hand, by induction hypothesis on (c), $U^{(L, C)}_{\omega}(t) \,\hat{u} \in F_{i - 1}^{(C)}(\theta_{t}\omega)$ for $t \ge 1$, consequently $\lim\limits_{t \to \infty} (1/t) \ln{\normC{U^{(L, C)}_{\omega}(t)\, \hat{u}}} = \lambda^{(C)}_{m'}$ for some $m' \ge i$.  Therefore
\begin{equation*}
\lambda^{(L)}_{i} = \lambda^{(C)}_{m'} \quad \text{for some } m' \ge i\,,
\end{equation*}
which together with~\eqref{eq:auxiliary-1}  give that $\lambda^{(L)}_{i} = \lambda^{(C)}_{i}$, as claimed.
\par\smallskip
Once the coincidence of the Lyapunov exponents $\lambda_i^{(C)}=\lambda_i^{(L)}$ is shown, the proof of (b) and (c) for the index  $i$ is completely analogous to the above proof for $i=1$ and it is omitted.
Similarly, parts (d) and (e) are due to the fact that for any $\omega \in \Omega_0$, $F_{\infty}^{(C)}(\omega)$ (resp.~$F_{\infty}^{(L)}(\omega)$) is characterized as the set of those $u \in C$ for which $\lim\limits_{t \to \infty} (1/t) \ln{\normC{U^{(C)}_{\omega}(t) \, u}} = - \infty$ (resp.~of those $\hat{u} \in L$ for which $\lim\limits_{t \to \infty} (1/t) \ln{\normL{U^{(L)}_{\omega}(t) \,\hat{u}}} = - \infty$), and again an application of Proposition~\ref{prop:Lyapunov-exponents-equal}.
\end{proof}
\begin{theorem}
\label{thm:E_i-equal}
Assume that both $\Phi^{(C)}$ and $\Phi^{(L)}$  admit an Oseledets decomposition. For any $i \in \{1, \ldots, k\}$ in case \textup{(I)}, or any $i \in \NN$ in case \textup{(II)} there holds:
\begin{equation*}
J(E^{(C)}_i(\omega)) = E^{(L)}_i(\omega),\text{ $\PP$-a.e.~on $\Omega$.}
\end{equation*}
\end{theorem}
\begin{proof}
Take $\w\in\Omega_0\subset \Omega_1$ and $u \in E_{i}^{(C)}(\omega) \setminus \{0\}$.
From Proposition~\ref{prop:full-orbit} let $\widetilde{u} \colon (-\infty, 0] \to C$ be a negative semiorbit for $\Phi^{(C)}$ with $\widetilde{u}(0) = u$ such that
\[
\lim_{s\to-\infty}\frac{\ln{\normC{\widetilde u(s)}}}{s}=\lambda_i\,.
\]
 It is straightforward that $\widehat{u} \colon (-\infty, 0] \to L$, $s\mapsto J(\widetilde u(s))$ is a negative semiorbit for $\Phi^{(L)}$ with $\widehat u(0) = Ju$. Moreover, from $\normL{\widehat u(s)}\leq 2\,\normC{\widetilde u(s)}$ we have
\begin{equation}\label{eq:liminf}
\liminf\limits_{s \to -\infty} \frac{\ln{\normL{\widehat u(s))}}}{s} \ge  \lambda_{i}.
\end{equation}
In addition, from $\widetilde u(s)(t)=z(t+1,\theta_{s-1}\omega,\widetilde u(s-1))$ for $t\in[-1,0]$
and Lemma~\ref{lm:estimate_pointwise} we deduce that
\begin{equation}\label{eq:normC-L}
\normC{\widetilde u(s)}\leq c(\theta_{s-1}\omega)(1+ d(\theta_{s-1}\omega)) \normL{\widehat u(s-1)}\,,
\end{equation}
which together with Lemma~\ref{lm:zero} provides
\begin{equation*}
\lambda_{i}\geq \limsup\limits_{s \to -\infty} \frac{\ln{\normL{\widehat u(s-1)}}}{s}=
\limsup\limits_{s \to -\infty} \frac{\ln{\normL{\widehat u(s)}}}{s} \,.
\end{equation*}
This inequality combined with~\eqref{eq:liminf} shows the existence of the limit
\begin{equation*}
\lim\limits_{s \to -\infty} \frac{\ln{\normL{\widehat u (s))}}}{s} = \lambda_{i}\,,
\end{equation*}
and since from Proposition~\ref{prop:Lyapunov-exponents-equal}(1) we have $\lim\limits_{t \to \infty} (1/t) \ln{\normL{U^{(L)}_{\omega}(t)\, (Ju)}}=\lambda_i$, we deduce that $Ju \in E^{(L)}_{i}(\omega)$, that is,  $J(E^{(C)}_{i}(\omega)) \subset E^{(L)}_{i}(\omega)$ is shown.\par
Next, take a nonzero vector $\hat u=(\hat u_1,\hat u_2)\in E_i^{(L)}(\w)$. Again, Proposition~\ref{prop:full-orbit} provides $\widehat u\colon (-\infty,0]\to L$ a negative semiorbit  for $\Phi^{(L)}$ with $\widehat u(0)=\hat u$ such that $\lim\limits_{s\to-\infty} (1/s)\ln{\normL{\widehat u(s)}}=\lambda_i$.
Since $\widehat u (s)=U_{\theta_{-1}\omega}^{(L)}(1)\,\widehat u(s-1)=(\widehat u_1(s),\widehat u_2(s))$,
we deduce that $\widehat u_2(s)\in C$, $J\widehat u_2(s)=\widehat u(s)$ for each $s\leq 0$, and  $\widetilde u\colon (-\infty,0]\to C$, $s\mapsto \widehat u_2(s)$ is a negative semiorbit for $\Phi^{(C)}$ with $\widetilde u(0)=\hat u_2$.
As before, from $\hat u\in E_i^{(L)}(\w)$, $\normL{\widehat u(s)}\leq 2\,\normC{\widetilde u(s)}$, the corresponding equality~\eqref{eq:normC-L} and Lemma~\ref{lm:zero} we deduce that
\[
\lambda_i\geq \limsup_{s\to-\infty}
\frac{\ln{\normC{\widetilde u(s)}}}{s}\geq \liminf_{s\to-\infty}\frac{\ln{\normC{\widetilde u(s)}}}{s}\geq \lambda_i\,, \text{ i.e. } \lambda_i=\lim\limits_{s\to-\infty} \frac{\ln{\normC{\widetilde u(s)}}}{s},
\]
 which together with $\lambda_i = \lim\limits_{t \to \infty} (1/t)\ln{\normC{U^{(L, C)}_{\omega}(t) \,\hat u}}= \lim\limits_{t \to \infty} (1/t) \ln{\normC{U^{(C)}_{\omega}(t) \,\hat u_2}}$,  show that  $\hat u_2\in E_i^{(C)}(\w)$, that is, $\hat u= J\hat u_2\in J(E_i^{(C)}(\w))$, which finishes the proof.
\end{proof}

\end{document}